\documentclass[11pt]{amsart}

\usepackage{amssymb,geometry}
\usepackage{color,enumitem,verbatim,graphicx}
\usepackage[colorlinks=true,urlcolor=blue,
citecolor=red,linkcolor=blue,linktocpage,pdfpagelabels,bookmarksnumbered,bookmarksopen]{hyperref}

\usepackage[hyperpageref]{backref}

\newtheorem{lm}{Lemma}[section]
\newtheorem{teo}[lm]{Theorem}
\newtheorem{prop}[lm]{Proposition}
\newtheorem{coro}[lm]{Corollary}

\theoremstyle{definition}
\newtheorem{oss}[lm]{Remark}
\newtheorem*{ack}{Acknowledgments}

\numberwithin{equation}{section}

\title[An inequality \`a la Szeg\H{o}-Weinberger on convex sets]{An inequality \`a la Szeg\H{o}-Weinberger for the $p-$Laplacian on convex sets}

\author[Brasco]{Lorenzo Brasco}
\author[Nitsch]{Carlo Nitsch}
\author[Trombetti]{Cristina Trombetti}

\address[L. Brasco]{Aix-Marseille Universit\'e, CNRS
	\newline\indent
	Centrale Marseille, I2M, UMR 7373, 39 Rue Fr\'ed\'eric Joliot Curie
	\newline\indent
	13453 Marseille, France}
\email{lorenzo.brasco@univ-amu.fr}

\address[C. Nitsch \& C. Trombetti]{Dipartimento di Matematica e Applicazioni ``R.Caccioppoli''
\newline\indent 
Complesso Universitario di Monte S. Angelo, Via Cintia
\newline\indent 
80126 Napoli, Italy}
\email{c.nitsch@unina.it}
\email{cristina@unina.it}

\begin{document}


\begin{abstract}
In this paper we prove a sharp upper bound for the first nontrivial eigenvalue of the $p-$Laplacian with Neumann boundary conditions. This applies to convex sets with given diameter. Some variants and extensions are investigated as well.
\end{abstract}

\maketitle

\begin{center}
	\begin{minipage}{11cm}
		\small
		\tableofcontents
	\end{minipage}
\end{center}

\section{Introduction}

\subsection{Overview}
Given an open  bounded connected Lipschitz set $\Omega\subset\mathbb{R}^N$ and an exponent $1<p<\infty$, we consider  $\lambda_p(\Omega)$ and $\mu_p(\Omega)$ the first nontrivial eigenvalues of the $p-$Laplace operator with Dirichlet and Neumann boundary conditions, respectively. We remind that these can be variationally characterized by
\[
\lambda_p(\Omega)=\min_{v\in W^{1,p}_0(\Omega)}\left\{\int_\Omega |\nabla v|^p\, dx\,:\, \int_\Omega |v|^p\, dx=1\right\},
\]
and
\[
\mu_{p}(\Omega)=\min_{v\in W^{1,p}(\Omega)}\left\{\int_\Omega |\nabla v|^p\, dx\, :\, \int_\Omega |v|^p\, dx=1 \mbox{ and }\int_\Omega |v|^{p-2}\, v\, dx=0\right\}.
\]
Since exact values of such quantities are known only for specific values of $p$ and special domains $\Omega$, it is important to give (sharp) estimates for these quantities 
in terms of (simple) geometric quantities such as measure, perimeter, diameter, relative isoperimetric constants and so on. In this direction, the reader could consult  
for instance \cite{A, AM, BCT, BCT2, FNT} and the references therein.
\par
With this respect the most celebrated example is the {\it Faber--Krahn inequality} (see \cite[Chapter 3]{He} for example) which asserts that
the following minimization problem
\begin{equation}
\label{FKmin}
\inf\{\lambda_{p}(\Omega)\, :\, |\Omega|\le c\},
\end{equation}
is (uniquely) solved by
$N-$dimensional balls of measure $c$. By taking advantage of the homegeneity properties of the functional $\Omega\mapsto \lambda_p(\Omega)$, the previous can be summarized as
\begin{equation}
\label{FK}
\lambda_p(\Omega)\ge \lambda_p(B)\,\left(\frac{|B|}{|\Omega|}\right)^\frac{p}{N},
\end{equation}
where $B$ is now any $N-$dimensional ball. Then $\lambda_p(\Omega)$ can be bounded from below in a sharp way just in terms the measure of the set $\Omega$. We point out that by using {\it isoperimetric inequality} or {\it isodiametric inequality}, from \eqref{FK} we can infer similar lower bounds for $\lambda_p$ in terms of the perimeter or the diameter of $\Omega$.  
\par
Observe that problem \eqref{FKmin} becomes trivial, when we replace $\lambda_p(\Omega)$ with $\mu_p(\Omega)$. Indeed, the latter is actually zero each time $\Omega$ is disconnected.
It turns out that the natural counterpart for $\mu_{p}$ is rather the maximization problem, i.e.
\begin{equation}
\label{SW}
\sup\{\mu_{p}(\Omega)\, :\, |\Omega|\ge c\}.
\end{equation}
Again, this is generally expected to be solved by $N-$dimensional balls of volume $c$. Unfortunately so far this problem has resisted all the attempts to be attacked with the unique exception of the case $p=2$ (and partially of the limiting cases $p=1$ and $p=\infty$, see \cite{EFKNT} and \cite{EKNT, RS} respectively). The {\it Szeg\H{o}--Weinberger inequality} \cite {Sz, We} states in fact that for $p=2$ problem \eqref{SW} is (uniquely) solved by $N-$dimensional balls of measure $c$. 
As before, the result can be rewritten in scaling invariant form as \begin{equation}
\label{SWscaling}
 \mu_{2}(\Omega)\le \mu_{2}(B) \left(\frac{|B|}{|\Omega|}\right)^{\frac{2}{N}},
\end{equation}
with equality holding if and only if $\Omega$ is an $N-$dimensional ball. We recall that the proof of \eqref{SW} for $p=2$ crucially exploits some pecularities of the Laplacian, like linearity and the knowledge of the explicit form of eigenfunctions on balls. 
\vskip.2cm
A couple of comments on the Szeg\H{o}-Weinberger result are in order. First of all, inequality \eqref{SWscaling} says that $\mu_2(\Omega)$ can be estimated {\it from above} just in terms of the measure of $\Omega$. But differently from the case $\lambda_p$, now {\it we can not directly infer} similar upper bounds for $\mu_2(\Omega)$ in terms of perimeter or diameter. Then one may wonder whether such a kind of estimates hold true or not for every $1<p<\infty$, at least for some particular classes of sets.
\par
Secondly, we notice that
if $B$ is any $N-$dimensional ball, using the fact that $\mu_{2}(B) < \lambda_{2}(B)$ (see \cite{He} or Proposition \ref{prop:nodal} below), from \eqref{SWscaling} we can infer
\begin{equation}
\label{mild}
 \mu_{2}(\Omega)\le \lambda_{2}(B) \left(\frac{|B|}{|\Omega|}\right)^{\frac{2}{N}},
\end{equation}
which can be see as a weak version of the Szeg\H{o}-Weinberger inequality. Again, a natural question is whether inequality \eqref{mild} can be extended to the case of $p\not=2$ or not.
\par
The last two questions are the starting point of our analysis. In this paper we prove indeed sharp upper bounds on $\mu_p(\Omega)$ in terms of diameter, as well as generalizations of \eqref{mild} for $p \neq 2$, under the additional constraint that $\Omega$ is a convex set. 
\subsection{A sharp upper bound} 
Then our main scope is to investigate the following shape optimization problem with convexity and diameter constraints
\begin{equation}
\label{mild2}
\mu^*:=\sup\Big\{\mu_{p}(\Omega)\, :\, \Omega\subset\mathbb{R}^N \mbox{ convex},\, \mathrm{diam}(\Omega)\ge 1\Big\}.
\end{equation}
Of course, by homogeneity of the quantities involved the value $1$ has no bearing and could be replaced by any constant $c>0$.
In Theorem \ref{p_cop} we show that the previous upper bound is finite, then we compute it and show at the same time that {\it this problem does not admit a solution}. Notably, we show that for every admissible set $\Omega$ there holds
\[
\mu_{p}(\Omega)<\mu^*,
\]
and we find a sequence $\{\Omega_n\}_{n\in\mathbb{N}}\subset\mathbb{R}^N$ of convex sets degenerating to a segment such that
\[
\lim_{n\to\infty} \mu_{p}(\Omega_n)=\mu^*.
\] 
We refer to Section \ref{sec:3} for more details.
As we will show, the previous result can be summarized by the following scaling invariant sharp inequality
\begin{equation}
\label{maestro}
\mu_{p}(\Omega)<\lambda_{p}(B)\,\left(\frac{\mathrm{diam}(B)}{\mathrm{diam}(\Omega)}\right)^{p},
\end{equation}
where $B$ is any $N-$dimensional ball and inequality sign is strict. The proof of \eqref{maestro} is based on a clever choice of a special test function which reminds ideas exploited in \cite{Ch}.
By joining \eqref{maestro} and the {\it isodiametric inequality}, we immediately get (see Corollary \ref{coro:weaksw} below)
\[
\mu_{p}(\Omega)<\lambda_{p}(B)\,\left(\frac{|B|}{|\Omega|}\right)^{\frac{p}{N}},
\]
which generalizes \eqref{mild} to $p\not=2$ for convex sets, as announced above. By keeping in mind the way such an estimate was proved for $p=2$, the previous can be seen as {\it the trace of a potentially existing Szeg\H{o}-Weinberger inequality for the $p-$Laplacian}. 
\par
For ease of completeness and in order to neatly motivate some of the studies performed in this paper, it is useful to recall at this point that the {\it minimization} problem
\begin{equation}
\label{mild3}
\inf\Big\{\mu_{p}(\Omega)\, :\, \Omega\subset\mathbb{R}^N \mbox{ convex},\, \mathrm{diam}(\Omega)\le 1\Big\},
\end{equation}
highlights the same features as problem \eqref{mild2}. For example, here as well the infimum can be computed and is not attained. More interestingly, a minimizing sequence is again given by a family of convex sets collapsing on a segment. For $p=2$ this is a celebrated result by Payne and Weinberger (see \cite{PW}), recently generalized in \cite{ENT,FNT} and \cite{Va} to $p\not=2$. The result can be summarized by the sharp inequality
\begin{equation}
\label{mild4}
\mu_{p}(\Omega)> \left(\frac{\pi_p}{\mathrm{diam}(\Omega)}\right)^p,
\end{equation}
where the constant $\pi_p$ is defined by
\begin{equation}
\label{pip}
\pi_p=2\,\int_0^{(p-1)^\frac{1}{p}} \left(1-\frac{t^p}{p-1}\right)^{-\frac{1}{p}}\, dt=2\,(p-1)^\frac{1}{p}\,\frac{\pi/p}{\sin(\pi/p)}.
\end{equation}
\subsection{Generalized eigenvalues} It is quite natural to wonder if similar conclusions can be drawn also in the case of the following generalized notion of first eigenvalues
\[
\lambda_{p,q}(\Omega)=\min_{v\in W^{1,p}_0(\Omega)}\left\{\int_\Omega |\nabla v|^p\, dx\, :\, \int_\Omega |v|^q\, dx=1\right\},
\]
and
\[
\mu_{p,q}(\Omega)=\min_{v\in W^{1,p}(\Omega)}\left\{\int_\Omega |\nabla v|^p\, dx\, :\, \int_\Omega |v|^q\, dx=1\,\mbox{ and }\,\int_\Omega |v|^{q-2}\, v\, dx=0\right\}.
\]
where $q\not =p$.
Quite interestingly, it turns out that for $q>p$ one has the following picture:
\begin{itemize}
\item one can prove the analogue of \eqref{maestro};
\vskip.2cm
\item this estimate {\it is not} sharp;
\vskip.2cm
\item the maximization problem 
\[
\sup\Big\{\mu_{p,q}(\Omega)\, :\, \Omega\subset\mathbb{R}^N \mbox{ convex},\, \mathrm{diam}(\Omega)\ge 1\Big\}
\]
now {\it admits a solution};
\vskip.2cm
\item a lower bound like \eqref{mild4} is not possible (and the infimum in \eqref{mild3} is $0$);
\end{itemize}
On the contrary, for $q<p$ all the previous statements have to be reverted. In particular, we have
\[
\sup\Big\{\mu_{p,q}(\Omega)\, :\, \Omega\subset\mathbb{R}^N \mbox{ convex},\ \mathrm{diam}(\Omega)\ge 1\Big\}=+\infty,
\] 
and it is rather the minimization problem for $\mu_{p,q}$ which is now well-posed (see Section \ref{sec:4} for more details).

\subsection{Plan of the paper}
In Section \ref{sec:2} we prove some basic results concerning properties of $\mu_{p,q}(\Omega)$ and $\lambda_{p,q}(\Omega)$. Section \ref{sec:3} is devoted to the investigation of problem 
\eqref{mild2}. In Section \ref{sec:4}  we consider the case $p\not = q$.
As a consequence of some the estimates proved in the paper, in Section \ref{sec:5} we exhibit a nodal domain property for Neumann eigenfunctions. Roughly speaking, this shows that for $q\ge p$ eigenfunctions associated to $\mu_{p,q}(\Omega)$ can not have a closed nodal line. Finally, Appendix \ref{sec:1d} deals with a one-dimensional variational problem and its extremals, whose properties are crucially exploited to prove optimality of the estimate \eqref{maestro}.
\begin{ack}
The first author has been partially supported by the Gaspard Monge Program for Optimization (PGMO), by EDF and the Jacques Hadamard Mathematical Foundation, through the research contract MACRO.
Part of this work has been done during some visits of the first author to Napoli. The Departement of Mathematics of the University of Napoli and its facilities are kindly acknowledged.
\end{ack}

\section{Preliminaries}
\label{sec:2}
We fix two exponents $p$ and $q$ such that $1<p<\infty$ and
\[
1<q<p^*:=\left\{\begin{array}{lr}
\displaystyle \frac{N\, p}{N-p},&\qquad \mbox{ if } 1<p<N,\\
&\\
+\infty,& \mbox{ if }p\ge N.\\
\end{array}
\right.
\]
For every $\Omega\subset\mathbb{R}^N$ open bounded Lipschitz set, we use the standard Sobolev spaces
\[
W^{1,p}(\Omega)=\big\{u\in L^p(\Omega)\, :\, \nabla u\in L^p(\Omega;\mathbb{R}^N)\big\},
\]
and $W^{1,p}_0(\Omega)$, the latter being the completion of $C^\infty_0(\Omega)$ with respect to the norm of $W^{1,p}(\Omega)$. We then define the two quantities
\[
\mu_{p,q}(\Omega)=\inf_{v\in W^{1,p}(\Omega)\setminus\{0\}}\left\{\frac{\displaystyle\int_\Omega |\nabla v|^p\, dx}{\displaystyle \left(\int_\Omega |v|^q\, dx\right)^\frac{p}{q}}\, :\, \int_\Omega |v|^{q-2}\, v\, dx=0\right\}
\]
and
\[
\lambda_{p,q}(\Omega)=\inf_{v\in W^{1,p}_0(\Omega)\setminus\{0\}} \frac{\displaystyle\int_\Omega |\nabla v|^p\, dx}{\displaystyle\left(\int_\Omega |v|^p\, dx\right)^\frac{p}{q}}.
\]
It is useful to recall that $\mu_{p,q}(\Omega)$ can be defined through the unconstrained minimization
\begin{equation}
\label{rallai}
\mu_{p,q}(\Omega)=\inf_{v\in \widehat W^{1,p}(\Omega)}\frac{\displaystyle\int_\Omega |\nabla v|^p\, dx}{\displaystyle \min_{t\in\mathbb{R}} \left(\int_\Omega |v-t|^q\, dx\right)^\frac{p}{q}},
\end{equation}
where we set $\widehat W^{1,p}(\Omega)=\{v\in W^{1,p}(\Omega)\, :\, \int_\Omega |\nabla v|^p\, dx>0\}$.
Also, we have that if $\lambda$ is such that the equation
\[
-\Delta_p u=\lambda\, \|u\|^{p-q}_{L^q(\Omega)}\, |u|^{q-2}\, u\quad \mbox{ in }\Omega,\qquad u=0,\quad \mbox{ on }\partial\Omega,
\]
admits a nontrivial solution in $W^{1,p}_0(\Omega)$, then $\lambda\ge \lambda_{p,q}(\Omega)$.
We start with a preliminary result on the quantities $\mu_{p,q}$ and $\lambda_{p,q}$.
\begin{lm}
\label{lm:comparison}
Let $1<p<\infty$ and $1<s<q<p^*$, then we have
\[
\lambda_{p,q}(\Omega)\le |\Omega|^{\frac{p}{s}-\frac{p}{q}}\, \lambda_{p,s}(\Omega)\qquad \mbox{ and }\qquad \mu_{p,q}(\Omega)\le |\Omega|^{\frac{p}{s}-\frac{p}{q}}\, \mu_{p,s}(\Omega).
\]
\end{lm}
\begin{proof}
The result is a plain consequence of H\"older inequality. Let us prove for example the second inequality: we pick $u_s\in \widehat W^{1,p}(\Omega)$ a function minimizing the Rayleigh quotient \eqref{rallai} which defines $\mu_{p,s}(\Omega)$.
We then define $t_q$ as the minimizer of
\[
t\mapsto \int_\Omega |u_s-t|^q\, dx,
\]
thus we get
\[
\begin{split}
\mu_{p,q}(\Omega)&=\min_{v\in \widehat W^{1,p}(\Omega)} \frac{\displaystyle\int_\Omega |\nabla v|^p\, dx}{\min\limits_{t\in\mathbb{R}}\displaystyle \left(\int_\Omega |v-t|^q\, dx\right)^\frac{p}{q}}\le \frac{\displaystyle\int_\Omega |\nabla u_s|^p\, dx}{\displaystyle \left(\int_\Omega |u_s-t_q|^q\, dx\right)^\frac{p}{q}}\\
&\le |\Omega|^{\frac{p}{s}-\frac{p}{q}}\, \frac{\displaystyle\int_\Omega |\nabla u_s|^p\, dx}{\displaystyle \left(\int_\Omega |u_s-t_q|^s\, dx\right)^\frac{p}{s}}\le |\Omega|^{\frac{p}{s}-\frac{p}{q}}\, \frac{\displaystyle\int_\Omega |\nabla u_s|^p\, dx}{\min\limits_{t\in\mathbb{R}}\displaystyle \left(\int_\Omega |u_s-t|^s\, dx\right)^\frac{p}{s}}
\end{split}
\]
which in turn gives the desired inequality, by minimality of $u_s$.
\end{proof}
We will also need the following very simple geometric result for convex sets.
\begin{lm}
\label{lm:normali}
Let $\Omega\subset\mathbb{R}^N$ be an open convex set, and let $x_0\in\partial\Omega$. Then 
\[
\langle x-x_0,\nu_\Omega(x)\rangle\ge 0,\qquad\mbox{ for $\mathcal{H}^{N-1}-$a.e. }x\in\partial\Omega,
\]
where $\nu_\Omega(x)$ denotes the outer unit normal to $\partial \Omega$ at the point $x$.
\end{lm}
\begin{proof}
Since $\Omega$ is convex, given $x\in\partial\Omega$ we have that
\[
\overline\Omega\subset \{y\in\mathbb{R}^N\, :\, \langle y-x,\nu_\Omega(x)\rangle\le 0\},
\]
i.e. the hyperplane orthogonal to $\nu_\Omega(x)$ and passing from $x$ is a supporting hyperplane for $\Omega$. In particular, since $x_0\in\overline\Omega$ we get $\langle x_0-x,\nu_\Omega(x)\rangle\le 0$, which concludes the proof.
\end{proof}

\section{A Szeg\H{o}-Weinberger inequality for convex sets}
\label{sec:3}

The following is the main result of the paper. This shows that the nonlinear spectral optimization problem
\[
\sup\Big\{\mu_{p,p}(\Omega)\, :\, \Omega\subset\mathbb{R}^N \mbox{ convex},\, \mathrm{diam}(\Omega)\ge 1\Big\},
\]
does not admit a solution, but a maximizing sequence is given by a family of convex sets suitably degenerating to a segment. Of course, the value $1$ for the diameter constraint plays no special role and could be replaced by any constant $c>0$.
\begin{teo}
\label{p_cop}
Let $\Omega\subset\mathbb{R}^N$ be an open bounded convex set and $1<p<\infty$. Then we have
\begin{equation}
\label{coppolicchio}
\mu_{p,p}(\Omega)<\lambda_{p,p}(B)\,\left(\frac{\mathrm{diam}(B)}{\mathrm{diam}(\Omega)}\right)^{p},
\end{equation}
where $B$ is any $N-$dimensional ball. 
\par
Equality sign in \eqref{coppolicchio} is never achieved but the inequality is sharp. More precisely, there exists a sequence $\{\Omega_k\}_{k\in\mathbb{N}}\subset\mathbb{R}^N$ of convex sets such that:
\begin{itemize}
\item $\mathrm{diam}(\Omega_k)=d>0$, for every $k\in\mathbb{N}$;
\item $\Omega_k$ converges to a segment of length $d$ in the Hausdorff topology;
\item we have
\begin{equation}
\label{figurine}
\lim_{k\to\infty}\mu_{p,p}(\Omega_k)=\lambda_{p,p}(B_{d/2}),
\end{equation}
where $B_{d/2}$ is an $N-$dimensional ball having radius $d/2$.
\end{itemize}
\end{teo}
\begin{proof}
We split the proof into two parts: at first we prove \eqref{coppolicchio}, then we construct the sequence $\{\Omega_k\}_{k\in\mathbb{N}}$ verifying \eqref{figurine}.
\vskip.2cm\noindent
\emph{Proof of \eqref{coppolicchio}.}  First of all, we observe that inequality \eqref{coppolicchio} is in scaling invariant form. Then without loss of generality, we can confine ourselves to prove that
\[
\mu_{p,p}(\Omega)<\lambda_{p,p}(B),
\] 
where $B$ is the ball centered at the origin such that $\mathrm{diam}(\Omega)=\mathrm{diam}(B)$. Let us take $u\in C^{1,\alpha}(\overline B)\cap C^\infty(B\setminus\{0\})$ the first Dirichlet eigenfunction for the ball $B$, normalized by the conditions
\[
\|u\|_{L^p(B)}=1\qquad \mbox{ and }\qquad u>0.
\]
This is radially symmetric and solves
\begin{equation}
\label{palla}
-\Delta_p u=\lambda_{p,p}(B)\, u^{p-1}\qquad \mbox{ and }\qquad u=0\quad \mbox{ on } \partial B.
\end{equation}
We then take two points $x_0,x_1\in\partial\Omega$ such that $|x_0-x_1|=\mathrm{diam}(\Omega)$,
and we define the two caps
\[
\Omega_i=\left\{x\, :\, |x-x_i|<\frac{\mathrm{diam}(\Omega)}{2}\right\}\cap \Omega,\qquad i=0,1,
\]
which are mutually disjoint (see Figure \ref{fig:fig3}).
\begin{figure}
\includegraphics[scale=.25]{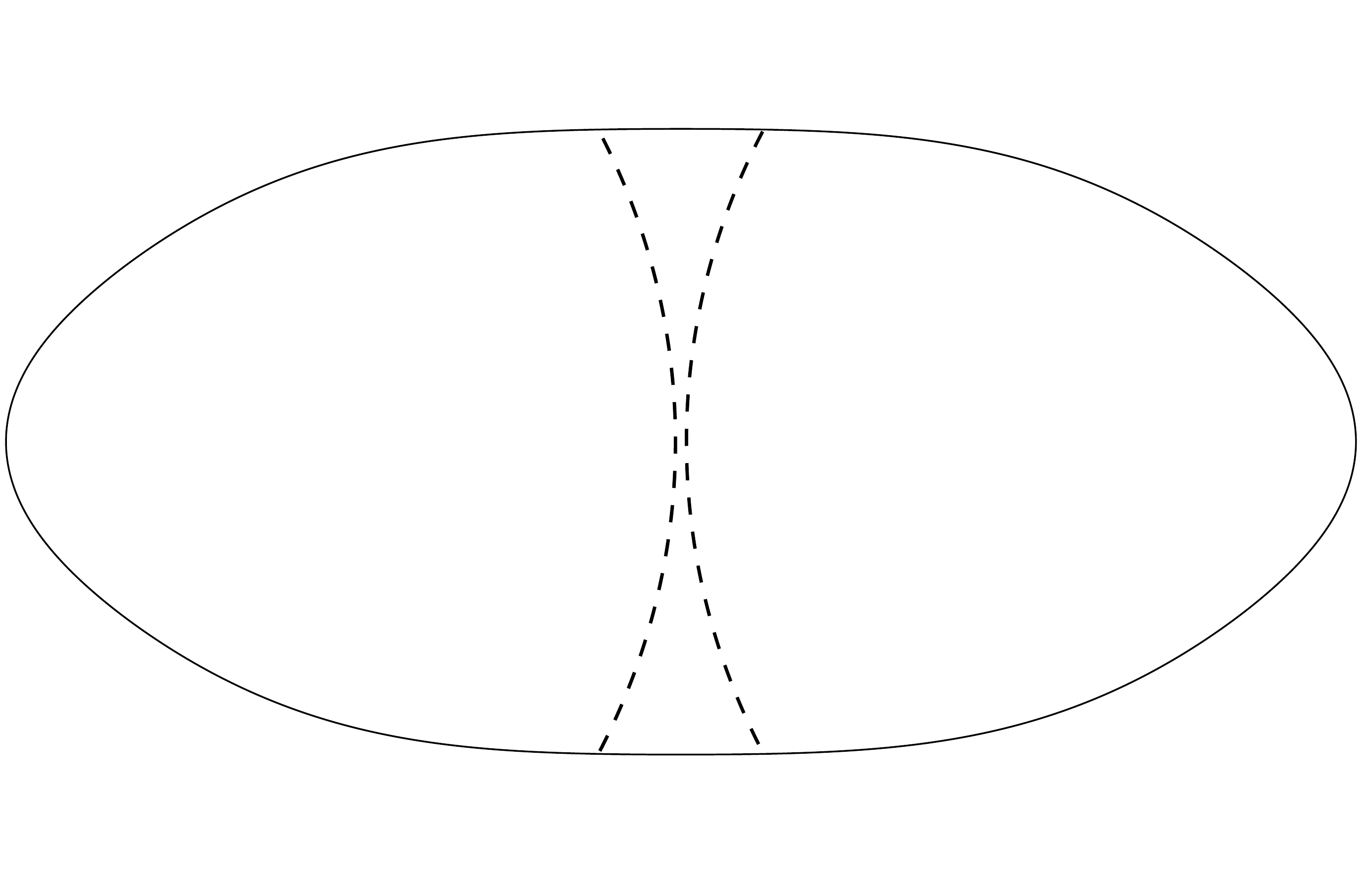}
\caption{The construction of the two caps $\Omega_0$ and $\Omega_1$.}
\label{fig:fig3}
\end{figure}
We then take the function
\[
\varphi(x)=u(x-x_0)\cdot 1_{\Omega_0}(x)-c\, u(x-x_1)\cdot 1_{\Omega_1}(x)\in W^{1,p}(\Omega),
\]
where the costant $c\in\mathbb{R}$ is given by
\[
c=\frac{\displaystyle\int_{\Omega_0} u(x-x_0)^{p-1}\, dx}{\displaystyle\int_{\Omega_1} u(x-x_1)^{p-1}\, dx},\qquad \mbox{ so that }\qquad \int_{\Omega} |\varphi|^{p-2}\, \varphi\, dx=0.
\]
By using this function $\varphi$ in the Rayleigh quotient defining $\mu_{p,p}(\Omega)$, we get
\[
\mu_{p,p}(\Omega)<
\frac{\displaystyle\int_{\Omega_0} |\nabla u(x-x_0)|^p\, dx+c^p\,\int_{\Omega_1} |\nabla u(x-x_1)|^p\, dx}{\displaystyle \int_{\Omega_0} |u(x-x_0)|^p\, dx+c^p \,\int_{\Omega_1} |u(x-x_1)|^q\, dx},
\]
where the strict inequality holds since $\varphi$ {\it can not be an eigenfunction}\footnote{Observe that if the Rayleigh quotient of $\varphi$ achieves the minimal value $\mu_{p,q}(\Omega)$, then $\varphi$ would solve
\[
-\Delta_p \varphi=\mu_{p,q}(\Omega)\, |\varphi|^{q-2}\, \varphi,\qquad \mbox{ in }\Omega,
\]
in a weak sense. Let us take $y_0\in \partial\Omega_0\cap \Omega$, by picking a ball $B_\varrho(y_0)$ with radius $\varrho$ sufficiently small so that $B_\varrho(y_0)\subset\Omega\setminus\Omega_1$, we would obtain that $\varphi$ is a nonnegative solution of the equation above in $B_\varrho(y_0)$. Then by Harnack's inequality (see \cite[Theorem 1.1]{Tr}) one obtains
\[
0<\max_{B_\varrho(y_0)} \varphi\le C\, \min_{B_\varrho(y_0)} \varphi=0,
\]
thus getting a contradiction. We point out that {\it we are not using any unique continuation argument}.}. 
By performing an integration by parts in the integrals at the numerator, we obtain\footnote{Observe that we have $u(x-x_0)=0$ on $\partial\Omega_0\cap\Omega$.}
\[
\begin{split}
\int_{\Omega_0} |\nabla u(x-x_0)|^p\, dx&=\int_{\partial\Omega\cap \partial\Omega_0} |\nabla u(x-x_0)|^{p-2}\, \frac{\partial u}{\partial\nu_\Omega}(x-x_0)\, u(x-x_0)\, d\mathcal{H}^{N-1}(x)\\
&-\int_{\Omega_0} \Delta_p u(x-x_0)\, u(x-x_0)\, dx\\
&=\int_{\partial\Omega\cap \partial\Omega_0} |\nabla u(x-x_0)|^{p-2}\, \frac{\partial u}{\partial\nu_\Omega}(x-x_0)\, u(x-x_0)\, d\mathcal{H}^{N-1}(x)\\
&+\lambda_{p,p}(B)\,\int_{\Omega_0} |u(x-x_0)|^p\, dx,
\end{split}
\]
where we used the equation \eqref{palla} solved by $u$. Observe that the first integral in the right-hand side is non-positive. Indeed $u$ is a radially decreasing function, then (with a small abuse of notation) we have
\[
\langle\nabla u(x-x_0),\nu_\Omega)(x)\rangle=u'(|x-x_0|)\, \left\langle \frac{x-x_0}{|x-x_0|},\nu_\Omega(x)\right\rangle,
\]
and the claim follows from Lemma \ref{lm:normali}, since $u'\le 0$. The same computations apply to the other terms appearing in the numerator, thus obtaining
\[
\mu_{p,p}(\Omega)< \lambda_{p,p}(B)\,\frac{\displaystyle \int_{\Omega_0} |u(x-x_0)|^q\, dx+c^p \,\int_{\Omega_1} |u(x-x_1)|^p\, dx}{\displaystyle\int_{\Omega_0} |u(x-x_0)|^p\, dx+c^p\,\int_{\Omega_1} |u(x-x_1)|^q\, dx}=\lambda_{p,p}(B),
\]
which concludes the proof of \eqref{coppolicchio}.
\vskip.2cm\noindent
\emph{Optimality of \eqref{coppolicchio}}. Let $B$ be a ball of diameter $d$. We now prove optimality of \eqref{coppolicchio}: for this we need to construct a sequence of convex sets $\{\Omega_k\}_{k\in\mathbb{N}}$, all sharing the same diameter $d$, and such that 
\begin{equation}
\label{optimality}
\lambda_{p,p}(B)\le \liminf_{k\to\infty} \mu_{p,p}(\Omega_k). 
\end{equation}
For all $s\in\mathbb{R}$ and $k\in \mathbb{N}\setminus\{0\}$ let us denote by 
\[
\mathcal{C}^-_k(s)=\left\{ (x_1,x')\in\mathbb{R}\times\mathbb{R}^{N-1}\, :\,  (x_1-s)_->k\,|x'| \right\}
\]  
and
\[
\mathcal{C}^+_k(s)=\left\{ (x_1,x')\in\mathbb{R}\times\mathbb{R}^{N-1}\, :\,  (x_1-s)_+>k\,|x'| \right\}
\]
the left and right circular infinite cone in $\mathbb{R}^N$ whose axis is the $x_1-$axis, having vertex in $(s,0)\in\mathbb{R}\times\mathbb{R}^{N-1}$, and whose opening angle is $\alpha=2\,\arctan(1/k)$. We set 
\[
\Omega_k=\mathcal{C}^-_k\left(\frac{d}{2}\right)\cap\mathcal{C}^+_k\left(-\frac{d}{2}\right).
\]
In dimension $N=2$, $\Omega_k$ is nothing but a rhombus of diagonals $d$ and $d/k$. In higher dimension $\Omega_k$ is obtained by gluing together the basis of two right circular cones of height $d/2$ and radii $d/(2\,k)$ (see Figure \ref{fig:fig2}).
\begin{figure}
\includegraphics[scale=.3]{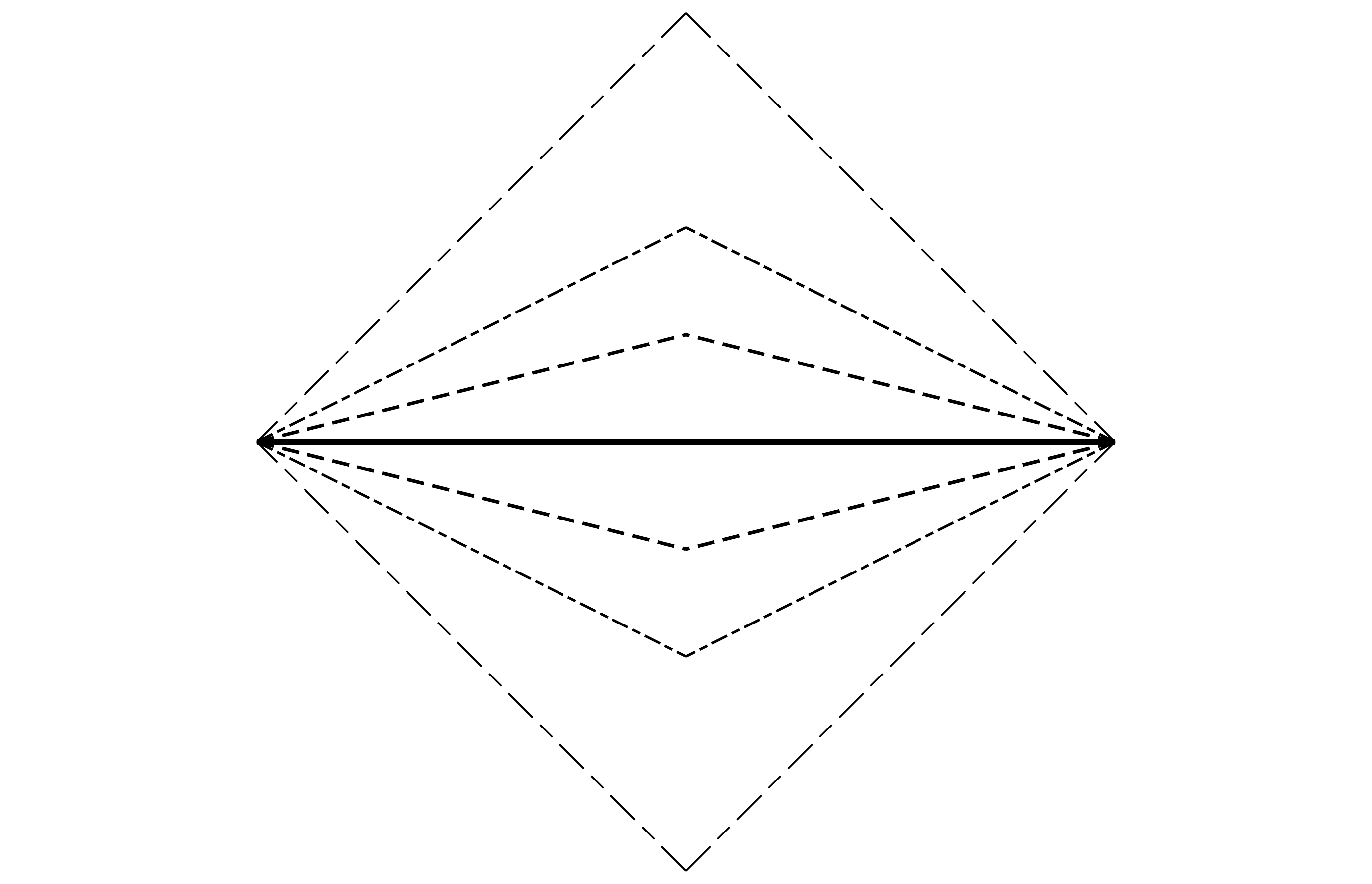}
\caption{The maximizing sequence $\Omega_k$ of Theorem \ref{p_cop}.}
\label{fig:fig2}
\end{figure}
\par
We claim that for this family inequality \eqref{optimality} holds true. 
We start observing that whenever $u\in W^{1,p}(\Omega_k)$ then the rescaled function $v(x_1,x')=u\left(x_1,x'/k\right)$ belongs to $W^{1,p}(\Omega_1)$ and we have
\[
\int_{\Omega_1} \left(|\partial_{x_1}v|^2+k^2|\nabla_{x'}v|^2\right)^\frac p2\, dx=k^{N-1}\,\int_{\Omega_k} |\nabla u|^p\, dx,\qquad \int_{\Omega_1} |v|^p\, dx =k^{N-1}\, \int_{\Omega_k} |u|^p\, dx,
\]
and
\[
\int_{\Omega_1} |v|^{p-2}\, v\, dx=k^{N-1}\, \int_{\Omega_k} |u|^{p-2}\, u\, dx=0.
\]
Thus we obtain 
\[
\begin{split}
\mu_{p,p}&(\Omega_k)=\min_{u\in W^{1,p}(\Omega_k)\setminus\{0\}}\left\{\frac{\displaystyle\int_{\Omega_k} |\nabla u|^p\, dx}{\displaystyle \int_{\Omega_k} |u|^p\, dx}\, :\, \int_{\Omega_k} |u|^{p-2}\, u\, dx=0\right\},\\
&=\min_{v\in W^{1,p}(\Omega_1)\setminus\{0\}}\left\{\frac{\displaystyle\int_{\Omega_1} \left( |\partial_{x_1}v|^2+k^2\,|\nabla_{x'}v|^2\right)^\frac p2\, dx}{\displaystyle \int_{\Omega_1} |v|^p\, dx}\, :\, \int_{\Omega_1} |v|^{p-2}\, v\, dx=0\right\}=:\gamma_k(\Omega_1).
\end{split}
\]
Now we denote by $u_k$ a function which minimizes the Rayleigh quotient defining $\mu_{p,p}(\Omega_k)$ and by $v_k(x_1,x')=u_k\left(x_1,x'/k\right)$ the corresponding function which minimizes the functional defining $\gamma_k(\Omega_1)$. Without loss of generality we can assume that $\|v_k\|_{L^p(\Omega_1)}=1$. Inequality \eqref{coppolicchio} implies that
\begin{equation}
\label{kappa}
\int_{\Omega_1} \left(  |\partial_{x_1}v_k|^2+k^2\,|\nabla_{x'}v_k|^2\right)^\frac{p}{2}\, dx\le C_{N,p,d},\qquad \mbox{ for all } k\in\mathbb{N}\setminus\{0\},
\end{equation}
then there exists $w\in W^{1,p}(\Omega_1)\setminus\{0\}$ so that $v_k\to w$ weakly in $W^{1,p}(\Omega_1)$ and strongly in $L^p(\Omega_1)$. Moreover we also have\footnote{The bound \eqref{kappa} implies that for every given $k_0\in\mathbb{N}\setminus\{0\}$, we have
\[
k_0^p\,\int_{\Omega_1} |\nabla_{x'} w|^p\, dx\le\int_{\Omega_1} \left(  |\nabla_{x_1} w|^2+k_0^2\,|\nabla_{x'}w|^2\right)^\frac p2\, dx\le\liminf_{k\to\infty}\int_\Omega \left(  |\nabla_{x_1}v_k|^2+k_0^2\,|\nabla_{x'}v_k|^2\right)^\frac{p}{2}\, dx\le C,
\]
which in turn gives $\nabla_{x'} w\equiv 0$ by the arbitrariness of $k_0$.} 
\[
\begin{split}
\nabla_{x'} w\equiv 0,\qquad \mbox{ and }\qquad\int_{\Omega_1} |w|^{p-2}\, w\, dx=0.
\end{split}
\]
Thus $w$ does not depend on the $x'$ variable and we will write for simplicity $w=w(x_1)$ with a slight abuse of notation. For all $s\in[-d/2,d/2]$ we denote by $\Gamma_s$ the section of $\Omega_1$ which is orthogonal to the $x_1-$axis at $x_1=s$ and set $g(s)=\mathcal{H}^{N-1}(\Gamma_s)$. Then we get
\[
\begin{split}
\liminf_{k\to\infty} \gamma_k(\Omega_1)&=\liminf_{k\to\infty}  \frac{\displaystyle\int_{\Omega_1} \left( |\partial_{x_1}v_k|^2+k^2|\nabla_{x'}v_k|^2\right)^\frac{p}{2}\, dx}{\displaystyle \int_{\Omega_1} |v_k|^p\, dx}\ge\liminf_{k\to\infty}  \frac{\displaystyle\int_{\Omega_1} |\partial_{x_1}v_k|^p\, dx}{\displaystyle \int_{\Omega_1} |v_k|^p\, dx}\\
&\ge \frac{\displaystyle\int_{\Omega_1} |w'|^p\, dx}{\displaystyle \int_{\Omega_1} |w|^p\, dx}=\frac{\displaystyle\int_{-d/2}^{d/2} |w'|^p \,g\, ds}{\displaystyle \int_{-d/2}^{d/2} |w|^p  \,g\, ds}\\
&\ge\min_{\phi\in W^{1,p}\left(\left(-d/2,d/2\right)\right)\setminus\{0\}}
\left\{\frac{\displaystyle\int_{-d/2}^{d/2} |\phi'|^p\,g\,\,ds}{\displaystyle \int_{-d/2}^{d/2} |\phi|^p\,g\,ds}\, :\, \int_{-d/2}^{d/2} |\phi|^{p-2}\,\phi\,g\, \,ds=0\right\}.
\end{split}
\]
Let us denote by $\eta$ the previous minimal value, then by Lemma \ref{lm:apres} a minimizer $f$ does exist and is a solution to the following boundary value problem
\[
\left\{\begin{array}{lc}
-\big(g\,|f'|^{p-2}\,f'\big)'=\eta \,g\,|f|^{p-2}\,f,& \mbox{ in } (-d/2,d/2),\\ &\\
f'(-d/2)=f'(d/2)=0.&
\end{array}
\right.
\] 
Still by Lemma \ref{lm:apres} we have that $f(0)=0$ and hence $f$ solves
\[
\left\{\begin{array}{lc}
-\big(g\,|f'|^{p-2}\,f'\big)'=\eta \,g\,|f|^{p-2}\,f,& \mbox{ in } (0,d/2),\\ &\\
f(0)=f'(d/2)=0.&
\end{array}
\right.
\] 
Finally, by reminding that $g(s)=\omega_{N-1}(d/2-s)^{N-1}$ for $0\le s\le d/2$, if we set $h(r)=f(d/2-r)$ then this solves
\[
\left\{
\begin{array}{lc}
-\big(r^{N-1}\,|h'|^{p-2}\,h'\big)'=\eta \, r^{N-1}\,|h|^{p-2}\, h, & \mbox{ in } (0,d/2), \\&\\
h'(0)=h(d/2)=0,&
\end{array}
\right.
\]
which means that the radial function $H(x)=h(|x|)$ is a Dirichlet eigenfunction of $-\Delta_p$ of a $N$-dimensional ball of radius $d/2$, namely $B$. Hence $\eta\ge \lambda_{p,p}(B)$ and we get
\[
\liminf_{k\to\infty} \mu_{p,p}(\Omega_k)=\liminf_{k\to\infty}\gamma_k(\Omega_1)\ge\lambda_{p,p}(B).
\]
This concludes the proof.
\end{proof}
From Theorem \ref{p_cop} and the isodiametric inequality
\begin{equation}
\label{isodiametric}
\frac{\mathrm{diam}(B)}{\mathrm{diam}(\Omega)}\le \left(\frac{|B|}{|\Omega|}\right)^\frac{1}{N},
\end{equation}
we can infer the following upper bound on $\mu_{p,p}$, in terms of the $N-$dimensional measure. We recall that for $p=2$ this is indeed a consequence of Szeg\H{o}-Weinberger inequality.
\begin{coro}
\label{coro:weaksw}
Let $\Omega\subset\mathbb{R}^N$ be an open bounded convex set and $1<p<\infty$. Then we have
\[
\mu_{p,p}(\Omega)<\lambda_{p,p}(B)\,\left(\frac{|B|}{|\Omega|}\right)^{\frac{p}{N}},
\]
where $B$ is any $N-$dimensional ball.
\end{coro}

\section{The case $p\not = q$}
\label{sec:4}

In this section we discuss variants and extensions of Theorem \ref{p_cop} for the quantity $\mu_{p,q}$ when $p\not =q$. 

\subsection{The case $p<q$} Actually, with the very same proof of Theorem \ref{p_cop} we can prove the following upper bound. This time, the resulting inequality {\it is not sharp} (see Remark \ref{oss:nonva!} below). For this reason, though the argument is the same, we prefer to give a separate statement.
\begin{teo}
\label{p_cop2}
Let $\Omega\subset\mathbb{R}^N$ be an open bounded convex set and $1<p<q<p^*$. Then we have
\begin{equation}
\label{coppolicchio2}
\mu_{p,q}(\Omega)<\lambda_{p,q}(B)\,\left(\frac{\mathrm{diam}(B)}{\mathrm{diam}(\Omega)}\right)^{p+\frac{N\,p}{q}-N},
\end{equation}
where $B$ is any $N-$dimensional ball.
\end{teo}
\begin{proof}
We use the same notation as in the proof of Theorem \ref{p_cop}. With $u\in C^{1,\alpha}(\overline B)\cap C^\infty(B\setminus\{0\})$ we now indicate the function achieving $\lambda_{p,q}(B)$, normalized by the conditions
\begin{equation}
\label{normalizzato}
\|u\|_{L^q(B)}=1\qquad \mbox{ and }\qquad u>0.
\end{equation}
By optimality it solves $-\Delta_p u=\lambda_{p,q}(B)\, u^{q-1}$ in $B$, with homogeneous Dirichlet boundary conditions.
As before, we consider the two caps $\Omega_0$ and $\Omega_1$ and take
\[
\varphi(x)=u(x-x_0)\cdot 1_{\Omega_0}(x)-c\, u(x-x_1)\cdot 1_{\Omega_1}(x)\in W^{1,p}(\Omega),
\]
where $c\in\mathbb{R}$ is chosen so to guarantee $\int_\Omega |\varphi|^{q-2}\,\varphi\,dx=0$.
By inserting $\varphi$ in the Rayleigh quotient defining $\mu_{p,q}(\Omega)$ and proceeding as in Theorem \ref{p_cop} we now end up with
\[
\mu_{p,q}(\Omega)< \lambda_{p,q}(B)\,\frac{\displaystyle \int_{\Omega_0} |u(x-x_0)|^q\, dx+c^p \,\int_{\Omega_1} |u(x-x_1)|^q\, dx}{\displaystyle\left(\int_{\Omega_0} |u(x-x_0)|^q\, dx+c^q\,\int_{\Omega_1} |u(x-x_1)|^q\, dx\right)^\frac{p}{q}}.
\]
The term on the right-hand side is of the form
\[
\frac{A+t^\frac{p}{q}\, B}{(A+t\,B)^\frac{p}{q}}.
\]
For $p<q$ the previous expression is maximal for $t=1$. Such a maximal value is given by $(A+B)^{1-p/q}$, we thus get
\begin{equation}
\label{restinooo}
\mu_{p,q}(\Omega)<\lambda_{p,q}(B)\,\left[\int_{\Omega_0} |u(x-x_0)|^q\, dx+\int_{\Omega_1} |u(x-x_1)|^q\, dx\right]^{1-\frac{p}{q}}.
\end{equation}
We have $1-p/q>0$ and the sum of the two terms into square brackets is less than $1$ by \eqref{normalizzato}, thus we can finally infer \eqref{coppolicchio2}.
\end{proof}
As in the case $p=q$, Theorem \ref{p_cop2} implies the following generalization of Corollary \ref{coro:weaksw}.
\begin{coro}
Let $\Omega\subset\mathbb{R}^N$ be an open bounded convex set and $1<p\le q<p^*$. Then we have
\[
\mu_{p,q}(\Omega)<\lambda_{p,q}(B)\,\left(\frac{|B|}{|\Omega|}\right)^{\frac{p}{N}+\frac{p}{q}-1},
\]
where $B$ is any $N-$dimensional ball.
\end{coro}

\begin{oss}[About sharpness]
\label{oss:nonva!}
This time, the estimate \eqref{coppolicchio2} is not sharp. We keep the same notation as in the proof of Theorem \ref{p_cop2} and still consider $q>p$. By adding and subtracting the term $\lambda_{p,q}(B)$ 
on the right-hand side of \eqref{restinooo}, recalling \eqref{normalizzato} and using the concavity of $t\mapsto t^{1-p/q}$, we get
\[
\begin{split}
\mu_{p,q}(\Omega)&<\lambda_{p,q}(B)+\lambda_{p,q}(B)\, \left[\left(\int_{\Omega_0\cup \Omega_1} |u|^q\, dx\right)^{1-\frac{p}{q}}-\left(\int_B |u|^q\, dx\right)^{1-\frac{p}{q}}\right]\\
&\le \lambda_{p,q}(B)+\frac{q-p}{q}\, \lambda_{p,q}(B)\, \left[\int_{\Omega_0\cup \Omega_1} |u|^q\, dx-\int_{B} |u|^q \, dx\right]\\
\end{split}
\]
Since $u$ is radially decreasing, a simple rearrangement argument finally gives
\begin{equation}
\label{restino0}
\mu_{p,q}(\Omega)<\lambda_{p,q}(B)-\frac{q-p}{q}\, \lambda_{p,q}(B)\,\int_{B\setminus T_\Omega} |u|^q\, dx
\end{equation} 
where $T_\Omega$ is the ball centered at the origin, such that $|T_\Omega|=|\Omega_0\cup\Omega_1|$. Now observe that by using the {\it quantitative isodiametric inequality} (see \cite[Theorem 1]{MPP})
\begin{equation}
\label{isodiam}
\begin{split}
|B\setminus T_\Omega|=|B|-|\Omega_0\cup \Omega_1|&=\Big(|B|-|\Omega|\Big)+\Big(|\Omega|-|\Omega_0\cup \Omega_1|\Big)\\
&\ge \frac{|\Omega|}{C_N}\, \mathcal{A}(\Omega)^2+\Big||\Omega|-|\Omega_0\cup\Omega_1|\Big|,
\end{split}
\end{equation}
where $C_N>0$ is a dimensional constant and $\mathcal{A}(\Omega)$ is the {\it Fraenkel asymmetry} of $\Omega$, defined by
\[
\mathcal{A}(\Omega)=\inf\left\{\frac{2\,|\Omega\setminus \Omega_\#|}{|\Omega_\#|}\, :\, \Omega_\# \mbox{ ball with } |\Omega_\#|=|\Omega|\right\}.
\]
Now suppose that there exists a sequence of convex sets $\{\Omega_n\}_{n\in\mathbb{N}}\subset\mathbb{R}^N$ such that
\[
\mathrm{diam}(\Omega_n)=\mathrm{diam}(B)\qquad \mbox{ and }\qquad \lim_{n\to\infty}\mu_{p,q}(\Omega_n)=\lambda_{p,q}(B),
\] 
for $q>p$. Then from \eqref{restino0} one would obtain
\[
\int_{B\setminus T_{\Omega_n}} |u|^q\, dx=0.
\]
Since $u>0$ in $B$, this would imply $|B\setminus T_{\Omega_n}|\to 0$ and thus from \eqref{isodiam}
\begin{equation}
\label{restino}
\lim_{n\to\infty}\mathcal{A}(\Omega_n)=0 \qquad \mbox{ and }\qquad \lim_{n\to\infty} \Big||\Omega_n|-|\Omega_{n,0}\cup\Omega_{n,1}|\Big|=0.
\end{equation}
The first condition in \eqref{restino} implies that $\Omega_n$ converges\footnote{In the $L^1$ sense, i.e. the characteristic functions $\{1_{\Omega_n}\}_{n\in\mathbb{N}}$ converge in $L^1(\mathbb{R}^N)$ to $1_B$.} to a ball, in contrast with the fact that $|\Omega|>|\Omega_0\cup\Omega_1|$ for a ball (see Figure \ref{fig:fig1}).
\begin{figure}
\includegraphics[scale=.25]{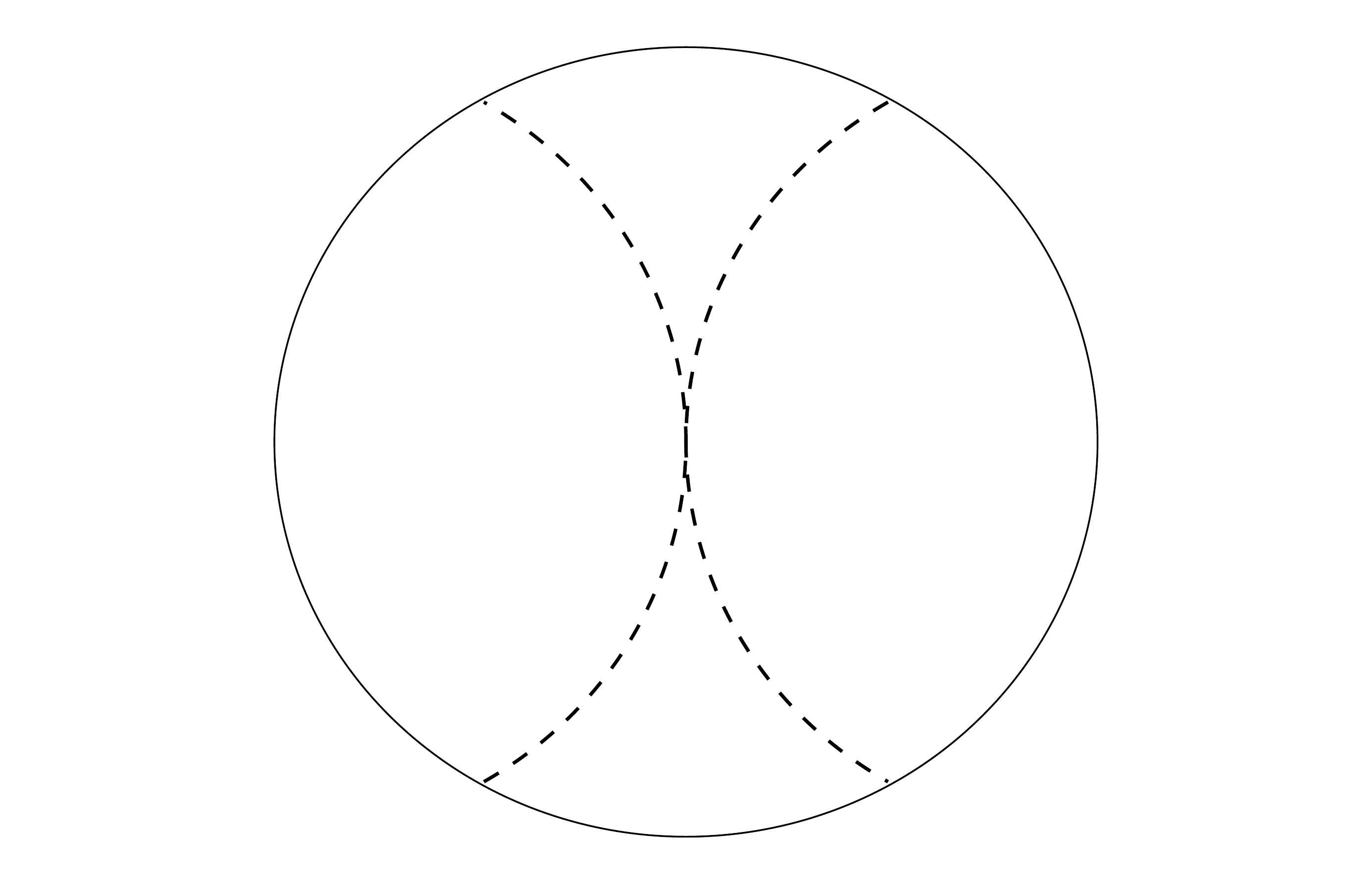}
\caption{The two caps $\Omega_0$ and $\Omega_1$ can not cover the whole ball.}
\label{fig:fig1}
\end{figure}
\end{oss}
As in the case $p=q$, we can then ask whether the following shape optimization problem
\begin{equation}
\label{shape}
\sup\Big\{\mu_{p,q}(\Omega)\, :\, \Omega \mbox{ open and bounded convex set},\, \mathrm{diam}(\Omega)\ge c\Big\},
\end{equation}
admits a solution or not. Quite surprisingly, this time we can infer existence of an optimal shape.
\begin{teo}[Existence of a maximizer]
\label{teo:esistenza}
Let $1<p<q<p^*$, for every $c>0$ problem \eqref{shape} admits a solution, i.e. there exists an open and bounded convex set $\mathcal{K}\subset\mathbb{R}^N$ such that
\[
\mu_{p,q}(\Omega)\, \Big(\mathrm{diam}(\Omega)\Big)^{p+\frac{N\,p}{q}-N}\le \mu_{p,q}(\mathcal{K})\, \Big(\mathrm{diam}(\mathcal{K})\Big)^{p+\frac{N\,p}{q}-N},
\]
for every $\Omega \subset\mathbb{R}^N$ open and bounded convex set.
\end{teo}
\begin{proof}
By Theorem \ref{p_cop2} we already know that the suprem \eqref{shape} is finite.
Let us call it $N_c$ and take a maximizing sequence of admissible sets $\{\Omega\}_{k\in\mathbb{N}}\subset\mathbb{R}^N$. Of course we can assume
\begin{equation}
\label{dalbasso}
\mu_{p,q}(\Omega_k)\ge \frac{N_c}{2}>0,\qquad \mbox{ for every }k\in\mathbb{N}.
\end{equation}
Since $\mu_{p,q}$ scales like a length to a negative power, we can also assume that
\[
\mathrm{diam}(\Omega_k)=c,\qquad \mbox{ for every }k\in\mathbb{N}.
\]
Finally, we can suppose that there exists a uniform constant $\delta>0$ such that
\begin{equation}
\label{salva}
|\Omega_k|\ge \delta,\qquad \mbox{ for every }k\in\mathbb{N},
\end{equation}
since otherwise we would have that $\mu_{p,q}(\Omega_k)$ goes to zero (see Remark \ref{oss:azzero} below).
\par
Thanks to the bound on the diameters, we can assume that the whole sequence $\{\Omega_k\}_{k\in\mathbb{N}}$ is contained in a common compact set $D\subset\mathbb{R}^N$. Thus the sequence is relatively compact for the complementary Hausdorff topology in $D$: more precisely, there exists an open set $\Omega\subset D$ such that $\Omega_k$ (up to a subsequence) converges in the Hausdorff complementary distance to $\Omega$ (see \cite[Corollaire 2.2.24]{HP}). Moreover, $\Omega$ is still convex and its diameter  equals $c$ (see \cite[Section 2.2.3]{HP}). We also observe that the characteristic functions $\{1_{\Omega_k}\}_{k\in\mathbb{N}}$ converges to $1_\Omega$ strongly\footnote{By convexity, a uniform bound on $\mathrm{diam}(\Omega_k)$ implies a uniform bound on their perimeters and measures. Then it is sufficient to use the compact embedding $BV(D)\hookrightarrow L^1(D)$, where $BV(D)$ is the space of functions with bounded variation.} in $L^1(D)$ and $\ast-$weakly in $L^\infty(D)$.
\par
Without loss of generality, we can assume that $\Omega$ contains the origin, since $\mu_{p,q}(\Omega)$ is not affected by translations.
We are now going to prove that
\begin{equation}
\label{scs}
\limsup_{k\to\infty} \mu_{p,q}(\Omega_k)\le \mu_{p,q}(\Omega).
\end{equation}
At this aim, let us take $u\in W^{1,p}(\Omega)$ a function attaining the infimum in the definition of $\mu_{p,q}(\Omega)>0$. Since $\Omega$ contains the origin, for every $\varepsilon>0$ the set $\Omega^\varepsilon:=(1+\varepsilon)\,\Omega$ is such that
\[
\Omega\Subset \Omega^\varepsilon.
\]
Then by Hausdorff convergence for every $\varepsilon>0$ there exists $k_\varepsilon\in\mathbb{N}$ such that
\[
\Omega_k\subset \Omega^\varepsilon,\qquad \mbox{ for every }k\ge k_\varepsilon.
\]
We also set
\[
u_\varepsilon(x)=u\left(\frac{x}{1+\varepsilon}\right),\qquad x\in \Omega^\varepsilon,
\]
then for every $0<\varepsilon<1$ and every $k\ge k_\varepsilon$, we take $t_{\varepsilon,k}\in\mathbb{R}$ such that
\[
\int_{\Omega_k} |u_\varepsilon-t_{\varepsilon,k}|^q\, dx=\min_{t\in\mathbb{R}}\int_{\Omega_k} |u_\varepsilon-t|^q\, dx,
\]
We claim that the sequence $\{t_{\varepsilon,k}\}_{k\in\mathbb{N}}$ in bounded uniformly in $k$ and $0<\varepsilon<1$, i.e. there exists $C>0$ such that 
\begin{equation}
\label{tt}
|t_{\varepsilon,k}|\le C,\qquad \mbox{ for every }\quad 0<\varepsilon<1\quad \mbox{ and }\quad k\ge k_\varepsilon.
\end{equation}
Indeed, observe that by convexity of the map $\tau\mapsto \tau^q$ and \eqref{salva}, we have
\[
\begin{split}
\int_{\Omega_k} |u_\varepsilon-t_{\varepsilon,k}|^q\, dx&\ge \frac{1}{2^{q-1}}\, |\Omega_k|\, |t_{\varepsilon,k}|^q-\int_{\Omega_k} |u_\varepsilon|^q\, dx\\
&\ge \frac{\delta}{2^{q-1}}\, |t_{\varepsilon,k}|^q-\int_{\Omega_\varepsilon} |u_\varepsilon|^q\, dx\\
&=\frac{\delta}{2^{q-1}}\, |t_{\varepsilon,k}|^q\,-(1+\varepsilon)^{N}\, \int_\Omega |u|^q\, dx,
\end{split}
\]
and on the other hand
\[
\begin{split}
\int_{\Omega_k} |u_\varepsilon-t_{\varepsilon,k}|^q\, dx&\le \frac{\displaystyle\int_{\Omega_k} |\nabla u_\varepsilon|^p\, dx}{\mu_{p,q}(\Omega_k)}\le 2\,\frac{(1+\varepsilon)^{N-p}}{N_c}\, \int_\Omega |\nabla u|^p\, dx,
\end{split}
\]
where we used \eqref{dalbasso} and the very definition fo $u_\varepsilon$.
By keeping the two estimates together, we finally get \eqref{tt}.
\par
Thus we can suppose that $t_{\varepsilon,k}$ converges (up to a subsequence) to $t_\varepsilon\in\mathbb{R}$ as $k$ goes to $\infty$, and $t_\varepsilon$ is in turn uniformly bounded. Then we get
\[
\limsup_{k\to\infty} \mu_{p,q}(\Omega_k)\le \limsup_{k\to\infty}\frac{\displaystyle\int_{\Omega_k} |\nabla u_\varepsilon|^p\, dx}{\displaystyle\left(\int_{\Omega_k} |u_\varepsilon-t_{\varepsilon,k}|^q\,dx \right)^{p/q}}\le \frac{\displaystyle\int_{\Omega} |\nabla u_\varepsilon|^p\, dx}{\displaystyle\left(\int_{\Omega} |u_\varepsilon-t_{\varepsilon}|^q\,dx \right)^{p/q}}
\]
for every $0<\varepsilon<1$, where we also used the $\ast-$weak convergence of the characteristic functions, recalled above. We now observe that 
\[
\lim_{\varepsilon\to 0} \int_\Omega |u_\varepsilon-t_\varepsilon|^q\, dx=\int_\Omega |u-\widetilde t|^q\, dx\ge \min_{t\in\mathbb{R}} \int_\Omega |u-t|^q\, dx,
\]
where $\widetilde t\in\mathbb{R}$ is an accumulation point of the net $\{t_\varepsilon\}_{\varepsilon>0}$, and also
\[
\lim_{\varepsilon\to 0} \|\nabla u_\varepsilon-\nabla u\|_{L^p(\Omega)}=0.
\]
Thus it is now sufficient to take the limit as $\varepsilon$ goes to $0$ in order to get \eqref{scs}, by arbitrariness of $u$. This finally gives that $\Omega$ is a solution of \eqref{shape}.
\end{proof}
\begin{oss}[Lower bounds and minimization]
\label{oss:azzero}
For $p<q$ the quantity $\mu_{p,q}(\Omega)$ can not be bounded {\it from below} in terms of $\mathrm{diam}(\Omega)$ only. In other words, for $q>p$ we have
\[
\inf\Big\{\mu_{p,q}(\Omega)\, :\, \Omega\subset\mathbb{R}^N \mbox{ convex},\ \mathrm{diam}(\Omega)\le c\Big\}=0.
\]
A minimizing sequence is given by any family of convex sets $\{\Omega_k\}_{k\in\mathbb{N}}\subset\mathbb{R}^N$ such that
\begin{equation}
\label{vanishing}
\lim_{k\to\infty} |\Omega_k|=0\qquad \mbox{ and }\qquad \mathrm{diam}(\Omega_k)=c.
\end{equation}
Indeed, by applying Lemma \ref{lm:comparison} we get 
\[
\mu_{p,q}(\Omega)\le |\Omega|^{1-\frac{p}{q}}\, \mu_{p,p}(\Omega).
\]
If we now apply Theorem \ref{p_cop} to the right-hand side, we gain
\[
\mu_{p,q}(\Omega)<\lambda_{p,p}(B)\,|\Omega|^{1-\frac{p}{q}}\,\left(\frac{\mathrm{diam}(B)}{\mathrm{diam}(\Omega)}\right)^{p}.
\]
Thus for a sequence of convex sets verifying \eqref{vanishing}, we get that $\mu_{p,q}(\Omega_k)$ converges to $0$.
\end{oss}

\subsection{The case $p>q$}

In this case, we can show that an upper bound on $\mu_{p,q}$ like that of \eqref{coppolicchio2} can not hold true and actually we have
\[
\sup\{\mu_{p,q}(\Omega)\, :\, \Omega\subset\mathbb{R}^N \mbox{ convex},\ \mathrm{diam}(\Omega)\ge c\}=+\infty.
\] 
Indeed, a maximizing sequence is given by any family of open convex sets $\{\Omega_k\}_{n\in\mathbb{N}}\subset\mathbb{R}^N$ such that
\[
\mathrm{diam}(\Omega_k)=c>0 \qquad \mbox{ and }\qquad \lim_{n\to\infty} |\Omega_k|=0.
\]
Actually, this is a consequence of estimate \eqref{op!} below.
\begin{prop}
Let $1<q<p$ and $\Omega\subset\mathbb{R}^N$ be an open and bounded convex set. Then we have
\begin{equation}
\label{op!}
\mu_{p,q}(\Omega)\ge \left(\frac{\pi_p}{\mathrm{diam}(\Omega)}\right)^p\, |\Omega|^{\frac{q}{p}-1},
\end{equation}
and
\begin{equation}
\label{controcoppolicchio}
\mu_{p,q}(\Omega)\ge \left(\frac{\pi_p}{|B|^{\frac{1}{q}-\frac{1}{p}}\,\mathrm{diam}(B)}\right)^p\,\left(\frac{\mathrm{diam}(B)}{\mathrm{diam}(\Omega)}\right)^{p+\frac{N\,p}{q}-N},
\end{equation}
where the constant $\pi_p$ is given by \eqref{pip}.
\end{prop}
\begin{proof}
Again by Lemma \ref{lm:comparison} with $s=p>q$, we get
\[
\mu_{p,p}(\Omega)\le |\Omega|^{1-\frac{q}{p}}\, \mu_{p,q}(\Omega).
\]
By using the lower bound \eqref{mild4}, we can obtain \eqref{op!}.
\vskip.2cm\noindent
Estimate \eqref{controcoppolicchio} is obtained by combining \eqref{op!} with the isodiametric inequality \eqref{isodiametric}.
\end{proof}
The estimate \eqref{controcoppolicchio} is the counterpart of Theorem \ref{p_cop2} for the case $q<p$. Thus this time it is the minimum problem
\[
\inf\Big\{\mu_{p,q}(\Omega)\, :\, \Omega\subset\mathbb{R}^N \mbox{ open and convex},\ \mathrm{diam}(\Omega)\le c\Big\},
\]
that actually makes sense. By suitably adapting the proof of Theorem \ref{teo:esistenza}, one can see that the previous problem admits indeed a solution. We leave the details to the interested reader.

\section{A nodal domain property}
\label{sec:5}

If $u$ is a function achieving the infimum in the problem defining $\mu_{p,q}(\Omega)$, then by {\it nodal domain} we mean every connected component of the (open) sets
\[
\{x\in\Omega\, :\, u(x)>0\}\qquad \mbox{ and }\qquad \{x\in\Omega\,: \, u(x)<0\}.
\]
As a consequence of Theorems \ref{p_cop} and \ref{p_cop2}, in the case $p\ge q$ we have the following result.
\begin{prop}
\label{prop:nodal}
Let $\Omega\subset\mathbb{R}^N$ be an open and bounded convex set and $1<p\le q<p^*$. Then 
\begin{equation}
\label{debole}
\mu_{p,q}(\Omega)<\lambda_{p,q}(\Omega).
\end{equation}
Moreover, every nodal domain of a function achieving $\mu_{p,q}(\Omega)$ has to intersect $\partial\Omega$.
\end{prop}
\begin{proof}
The proof of \eqref{debole} immediately follows by combining \eqref{coppolicchio}, the Faber-Krahn inequality
\[
|B|^{\frac{p}{q}+\frac{p}{N}-1}\, \lambda_{p,q}(B)\le |\Omega|^{\frac{p}{q}+\frac{p}{N}-1}\, \lambda_{p,q}(\Omega)
\]
and the isodiametric inequality.
\par
To prove the second assertion, let us argue by contradiction. We take $v$ achieving $\mu_{p,q}(\Omega)$ and we assume that the open set $\{x\in\Omega\, :\, v>0\}$ has a connected component $\omega\Subset \Omega$. We can further suppose that $\|v\|_{L^q(\Omega)}=1$,
then $v\in W^{1,p}_0(\omega)$ and it solves
\[
-\Delta_p v=\mu_{p,q}(\Omega)\, v^{q-1},\qquad \mbox{ in }\omega,
\]
so that
\[
\int_{\omega} |\nabla v|^p=\mu_{p,q}(\Omega)\, \int_\omega |v|^q\, dx\le \mu_{p,q}(\Omega)\, \left(\int_{\omega} |v|^{q}\, dx\right)^\frac{p}{q},
\]
thanks to the fact that $1=\|v\|_{L^q(\Omega)}\ge \|v\|_{L^q(\omega)}$ and $p/q\le 1$. This yields $\lambda_{p,q}(\omega)\le \mu_{p,q}(\Omega)$.
By using the strict monotonicity of $\lambda_{p,q}(\Omega)$ with respect to set inclusion and \eqref{debole}, we then get
\[
\lambda_{p,q}(\Omega)<\lambda_{p,q}(\omega)\le \mu_{p,q}(\Omega)<\lambda_{p,q}(\Omega),
\]
which gives the desired contradiction.
\end{proof}
\begin{oss}
When $p=q=2$, the previous argument to infer that first nontrivial Neumann eigenfunctions can not have a closed nodal line was originally due to Pleijel (see \cite{Pl}). For the Laplacian, inequality \eqref{debole} was conjectured by Kornhauser and Stakgold (see \cite{KS}) and can be obtained (again) as a byproduct of the Szeg\H{o}--Weinberger inequality \eqref{SWscaling}.
\end{oss}

\appendix

\section{A one-dimensional problem}

\label{sec:1d}

Let $d>0$ and let $g:[-d/2,d/2]\to\mathbb{R}^+$ be the Lipschitz continuous function defined by
\[
g(s)=\omega_{N-1}\,\left|\frac{d}{2}-s\right|^{N-1},
\]
as in the proof of Theorem \ref{p_cop}. We consider the variational problem
\begin{equation}
\label{variational}
\eta:=\inf_{v\in W^{1,p}\left(\left(-d/2,d/2\right)\right)\setminus\{0\}}
\left\{\frac{\displaystyle\int_{-d/2}^{d/2} |v'|^p\,g\,\,ds}{\displaystyle \int_{-d/2}^{d/2} |v|^p\,g\,ds}\, :\, \int_{-d/2}^{d/2} |v|^{p-2}\,v\,g\, \,ds=0\right\}.
\end{equation}
In this section, we state and prove some properties of extremals of \eqref{variational}, needed to prove the sharpness of estimate \eqref{coppolicchio}. 
\begin{lm}
\label{lm:apres}
With the notation above, we have $\eta>0$ and problem \eqref{variational} admits a solution.
Any optimizer $f$ is a weak solution of 
\begin{equation}
\label{ODE}
\left\{\begin{array}{lc}
-\big(g\,|f'|^{p-2}\,f'\big)'=\eta \,g\,|f|^{p-2}\,f,& \mbox{ in } (-d/2,d/2),\\ &\\
f'(-d/2)=f'(d/2)=0.&
\end{array}
\right.
\end{equation}
Moreover, $f$ vanishes at $x=0$ only and thus is also a weak solution of
\[
\left\{\begin{array}{lc}
-\big(g\,|f'|^{p-2}\,f'\big)'=\eta \,g\,|f|^{p-2}\,f,& \mbox{ in } (0,d/2),\\ &\\
f(0)=f'(d/2)=0.&
\end{array}
\right.
\] 
\end{lm}
\begin{proof}
In order to prove existence, we define as in the proof of Theorem \ref{p_cop} the convex set 
$\Omega_1\subset\mathbb{R}^N$, obtained by gluing together the basis of two right circular cones of height $d/2$ and radii $d/2$ (see Figure \ref{fig:fig2}). Then we recall that the function $g(s)$ coincides with the measure of the $(N-1)-$dimensional section
\[
\Omega_1\cap \{x_1=s\}.
\]
We consider the Poincar\'e-type constant
\begin{equation}
\label{modificato}
\widehat\eta:=\inf_{v\in W^{1,p}_{\mathbf{e}_1}(\Omega_1)\setminus\{0\}}
\left\{\frac{\displaystyle\int_{\Omega_1} |\nabla v|^p\,\,dx}{\displaystyle \int_{\Omega_1} |v|^p\,dx}\, :\, \int_{\Omega_1} |v|^{p-2}\,v\,dx=0\right\},
\end{equation}
where $W^{1,p}_{\mathbf{e}_1}(\Omega_1)$ is the closed subspace of $W^{1,p}(\Omega_1)$, built up of functions depending on the first variable $x_1$ only. By a standard comptacness argument, it is easy to see that $\widehat\eta>0$ and the infimum is attained. Moreover, by construction, we obtain that $\widehat \eta=\eta$ and the restiction to $x_1-$axis of any minimizer $u$ of \eqref{modificato} minimizes \eqref{variational} as well.
\vskip.2cm\noindent
We now prove the claimed properties of extremals. We begin by noticing that any solution $f$ of \eqref{variational} solves as well
\begin{equation}
\label{variationalbis}
\min_{v\in\mathcal{A}}\left\{\frac{1}{p}\,\int_{-d/2}^{d/2} |v'|^p\,g\,\,ds-\frac{\eta}{p}\,\int_{-d/2}^{d/2} |v|^p\,g\,ds\right\},
\end{equation}
where the set $\mathcal{A}$ is given by 
\[
\mathcal{A}=\left\{v\in W^{1,p}((-d/2,d/2))\, :\, \int_{-d/2}^{d/2} |v|^{p-2}\,v\,g\, \,ds=0\right\}.
\]
For $p\ge 2$, we observe that $\mathcal{A}$ is a $C^1$ manifold, thus we can apply the Lagrange Multipliers Theorem. This yields that $f$ has to satisfy
\[
\begin{split}
\int_{-d/2}^{d/2} |f'|^{p-2}\, f'\, \varphi'\, g\,ds&-\eta\, \int_{-d/2}^{d/2} |f|^{p-2}\, f\, \varphi\,g\, ds\\
&+\mu\,\int_{-d/2}^{d/2} |f|^{p-2}\, \varphi\, g\, ds=0,\quad \mbox{ for every } \varphi\in W^{1,p}((-d/2,d/2)),
\end{split}
\]
for some multiplier $\mu\in\mathbb{R}$. By choosing $\varphi\equiv 1$ and by using that $f\in\mathcal{A}$, we can conclude that $\mu=0$, i.e. $f$ is a weak solution of \eqref{ODE}.
\par
For $1<p<2$ some care is needed, since this time $\mathcal{A}$ is no more smooth. However, by proceeding as in \cite{DGS} (see also \cite[Lemma 5.8]{BF_nodea}), we can prove again that $f$ has to be a weak solution of \eqref{ODE}.
\vskip.2cm\noindent
We now prove that $f(x)=0$ if and only if $x =0$. Since $f$ is admissible in \eqref{variational} and $g$ is positive, $f$ has to vanish somewhere in the interval $(-d/2,d/2)$. Let us suppose that $f(x_0)=0$, with $x_0\in(0,d/2)$. Then we consider the new function
\[
\widetilde f(x)=\left\{\begin{array}{cl}
\displaystyle\int_{x_0}^x |f'(s)|\,ds,& \mbox{ if } x_0<x<d/2,\\
0,& \mbox{ if } -x_0\le x\le x_0\\
-\displaystyle\int_{x_0}^{-x} |f'(s)|\,ds,& \mbox{ if } -d/2<x<-x_0.
\end{array}
\right.
\]
This function is still admissible in \eqref{variational}, thanks to the symmetry of $g$. Moreover, by construction $\widetilde f$ is non-decreasing and such that
\[
\widetilde f'(x)=|f'(x)|,\quad x\in(x_0,d/2)\qquad \mbox{ and }\qquad \widetilde f(x)\ge |f(x)|,\quad x\in (x_0,d/2),
\]
thanks to the fact that $f(x_0)=0$. By using the equation \eqref{ODE} in weak form and testing it against $f\cdot 1_{(x_0,d/2)}$, we get
\[
\int_{x_0}^{d/2} |\widetilde f'|^p\,g\,ds=\int_{x_0}^{d/2} |f'|^{p}\, g\,ds=\eta\, \int_{x_0}^{d/2} |f|^{p}\,g\, ds\le\eta\,\int_{x_0}^{d/2} |\widetilde f|^p\,g\,ds.
\]
By the properties of $\widetilde f$ and eveness of $g$, we thus obtain
\[
\int_{-d/2}^{d/2} |\widetilde f'|^{p}\, g\,ds=2\, \int_{x_0}^{d/2} |\widetilde f'|^p\,g\,ds\le2\,\eta\,\int_{x_0}^{d/2} |\widetilde f|^p\,g\,ds=\eta\,\int_{-d/2}^{d/2} |\widetilde f|^p\,g\,ds.
\]
This shows that $\widetilde f$ still solves \eqref{variational}. In order to reach a contradiction, for all $\varepsilon>0$ we can define the increasing function
\[
f_\varepsilon(x)=\left\{\begin{array}{cl}
\widetilde f(x)+\varepsilon\, x_0,& \mbox{ if } x_0<x<d/2,\\
\varepsilon\, x,& \mbox{ if } -x_0\le x\le x_0,\\
\widetilde f(x)-\varepsilon\, x_0,& \mbox{ if } -d/2<x<-x_0.
\end{array}
\right.
\]
By symmetry $f_\epsilon$ is still an admissible function in \eqref{variational}. By using the convexity inequality
\[
|a+\varepsilon\, b|^p \ge |a|^p+p\,|a|^{p-2}\,a\,\varepsilon\, b,
\]
we have
\[
\frac{\displaystyle\int_{-d/2}^{d/2} |f'_\varepsilon|^p\, g\,ds}{\displaystyle\int_{-d/2}^{d/2} |f_\varepsilon|^p\, g\,ds}=\frac{\displaystyle\int_{0}^{d/2} |f'_\varepsilon|^p g\,ds}{\displaystyle\int_{0}^{d/2} |f_\varepsilon|^p\, g\,ds}\le
\dfrac{\displaystyle\int_{-d/2}^{d/2} |\widetilde f'|^p\, g\,ds+2\,\varepsilon^p\displaystyle\int_{0}^{d/2} g\,ds}{\displaystyle\int_{-d/2}^{d/2} |\widetilde f|^p g\,ds + 2\,\varepsilon\,  p\, x_0\, \int_{0}^{d/2} \widetilde f^{p-1}\,g\,ds}.
\]
For $\varepsilon$ small enough, this contradicts the optimality of $\widetilde f$ in \eqref{variational}, since $x_0>0$ by assumption. Thus we get $x_0=0$ and $f$ vanishes at the origin only.
\end{proof}

\end{document}